\documentclass[12pt]{amsart}
\usepackage{amssymb}
\usepackage{amsthm}
\usepackage{amsmath}
\usepackage{enumerate}
\usepackage{graphicx}
\usepackage{mathrsfs}
\usepackage{xypic}
\usepackage{color}

\usepackage[latin1]{inputenc}

\newcommand{\NN}{\mathbb{N}}

\newcommand{\D}{\mathcal D}

\newcommand{\union}{\cup}

\newcommand{\Map}{\mathrm{Map}}

\theoremstyle{plain}

\newtheorem{theorem}{Theorem}[section]
\newtheorem*{theorem*}{Theorem}
\newtheorem{lemma}[theorem]{Lemma}

\newtheorem{cor}[theorem]{Corollary}

\theoremstyle{definition}
\newtheorem{defi}[theorem]{Definition}

\theoremstyle{remark}

\newtheorem{example}[theorem]{Example}

\begin{document}
\title{Isoperimetric Inequalities for the handlebody groups}
\author{Ursula Hamenst\"adt and Sebastian Hensel}
\date{\today}
\thanks{
  Both authors are partially
  supported by the Hausdorff Center Bonn and the Hausdorff Institut
  Bonn. The second author is supported by the Max-Planck Institut
  f\"ur Mathematik Bonn}
\begin{abstract}
  We show that the mapping class group of a handlebody $V$ of genus at
  least $2$ has a Dehn function of at most exponential growth type.
\end{abstract}
\address{\hskip-\parindent
  Ursula Hamenst\"adt and Sebastian Hensel\\ Mathematisches Institut der Universit\"at Bonn \\
  Endenicher Allee 60\\
  D-53115 Bonn, Germany}
\email{ursula@math.uni-bonn.de}
\email{loplop@math.uni-bonn.de}
\maketitle

\section{Introduction}
A handlebody of genus $g\geq 2$ is a compact orientable $3$-manifold $V$
whose boundary $\partial V$ is a closed surface of genus $g$.
The \emph{handlebody group} ${\rm Map}(V)$ is the group of isotopy classes of
orientation preserving homeomorphisms of $V$. 
Via the natural restriction homomorphism, the group ${\rm Map}(V)$ 
can be viewed as a subgroup of the mapping class
group ${\rm Map}(\partial V)$ of $\partial V$. 
This subgroup is of infinite index, and it surjects onto the
outer automorphism group of the fundamental group of $V$
which is the free group with $g$ generators.

The handlebody group is finitely presented. Thus $\mathrm{Map}(V)$ can
be equipped with a word norm that is unique up to quasi-isometry.
Hence, the handlebody group carries a well-defined large-scale
geometry. However, this large scale geometry is not
compatible with the large-scale geometry of the
ambient group ${\rm Map}(\partial V)$. Namely, we showed in 
\cite{HH11} that the handlebody group is an exponentially
  distorted subgroup of the mapping class group of the boundary
  surface for every genus $g \geq 2$. 
Here, a finitely generated subgroup $H < G$ of a finitely generated
group $G$ is called \emph{exponentially distorted} if the following holds.
First, the word norm in $H$ of every element $h\in H$ can be bounded
from above by an exponential function 
in the word norm of $h$ in $G$. On the other hand, there
is no such bound with sub-exponential growth rate.

As a consequence, it is not possible to directly transfer geometric
properties from the mapping class group to the handlebody group.
In this note we initiate an investigation of the intrinsic large-scale
geometry of the handlebody group.

A particularly useful geometric invariant of a finitely
presented group $G$ is its \emph{Dehn function}, which 
can be defined as the isoperimetric function of a presentation complex
for $G$ (see Section~\ref{s:upperbound-and-dehn-function} for a
complete definition).
Although the Dehn function itself depends on the choice of a finite
presentation of $G$, the growth type of the Dehn function does not.
In fact, the growth type of the Dehn function is a quasi-isometry
invariant of $G$.

The mapping class group ${\rm Map}(\partial V)$ of $\partial V$
is automatic \cite{Mo95} and hence has quadratic Dehn function.
Since ${\rm Map}(V)$ is exponentially distorted
in ${\rm Map}(\partial V)$, this fact does not
provide any information on the Dehn function of ${\rm Map}(V)$.
On the other hand, for $g\geq 3$ 
the Dehn function of the outer automorphism group ${\rm Out}(F_g)$ of a 
free group $F_g$ on $g$ generators is exponential
\cite{HV96,BV95,BV10,HM10}. However, since the kernel of the
projection from the handlebody group to $\mathrm{Out}(F_g)$ is
infinitely generated \cite{McC85}, this fact also does not restrict
the Dehn function of $\mathrm{Map}(V)$.

The goal of this note is to give an upper bound for the
Dehn function of ${\rm Map}(V)$. We show
\begin{theorem}\label{dehn}
The handlebody group ${\rm Map}(V)$ satisfies an ex\-po\-nen\-tial
iso\-peri\-metric inequality, i.e. the growth of its
Dehn function is at most exponential.
\end{theorem}

The strategy of proof for Theorem~\ref{dehn} is similar
to the strategy employed in \cite{HV96} to show an exponential upper
bound for the Dehn function of outer automorphism groups of free groups. 
We construct a graph which is a geometric model for the handlebody
group (a similar construction is used in \cite{HH11} in order 
to show exponential distortion of handlebody groups). 
Vertices of this graph correspond to isotopy classes of 
special cell decompositions of $\partial V$ containing the boundary
of a \emph{simple disk system} in their one-skeleton. 
A simple disk system is a collection of pairwise disjoint,
pairwise non-homotopic embedded disks in $V$ which decompose $V$ 
into simply connected regions.
We then use a a surgery
procedure for disk systems to define a distinguished class of paths in
this geometric model.
Although these paths are in general not quasi-geodesics for 
the handlebody group (see the example at the end of this note),
they are sufficiently well-behaved so that they can be
used to fill a cycle with area bounded by an exponential
function in the length of the cycle.

The organization of this note is as follows.
In Section 2 we introduce disk systems and special 
paths in the disk system graph. Section 3 discusses
a geometric model for the handlebody group. 
This model is used in Section 4 for the proof of Theorem~\ref{dehn}.

\section{Disk exchange paths}
\label{sec:disk-systems}
In this section we collect some facts about properly embedded disks in
a handlebody $V$ of genus $g\geq 2$.
In particular, we describe a surgery procedure that is
central to the construction of paths in the handlebody group.

\medskip
A disk $D$ in $V$ is called \emph{essential} if it is properly
embedded and if $\partial D$ is an essential simple closed curve on
$\partial V$.
A \emph{disk system} for $V$ is a set of pairwise disjoint essential
disks in $V$ no two of which are homotopic.
A disk system is called \emph{simple} if all of its complementary
components are simply connected. 
It is called \emph{reduced} if in
addition it has a single complementary component.   

We usually consider disks and disk systems only up to proper isotopy.
Furthermore, we will always assume that disks and disk systems are in
minimal position if they intersect. Here we say that two disk systems
$\D_1,\D_2$ are in \emph{minimal position} if their boundary multicurves 
intersect in the minimal number of points in their respective isotopy
classes and if every component of $\D_1\cap
\D_2$ is an embedded arc in $\D_1\cap \D_2$ 
with endpoints in $\partial \D_1 \cap \partial \D_2$.
Note that minimal position of disks is not unique; in particular the
intersection pattern $\D_1 \cap \D_2$ is not determined by the isotopy
classes of $\D_1$ and $\D_2$.

The following easy fact will be used frequently
throughout the article.
\begin{lemma}\label{lemma:han-action-on-disk-sys}
  The handlebody group acts transitively on the set of isotopy
  classes of reduced disk systems. 
  Every mapping class of $\partial V$ that fixes the isotopy class of
  a simple disk system is contained in the handlebody group.
\end{lemma}
\begin{proof}
  The first claim follows from the fact that the complement of a
  reduced disk system in $V$ is a ball with $2g$ spots and any two
  such manifolds are homeomorphic. 
  The second claim is immediate since every homeomorphism of the
  boundary of a spotted ball extends to the interior.
\end{proof}

Let $\D$ be a disk system. An \emph{arc relative to
  $\D$} is a continuous embedding $\rho:[0,1]\to\partial V$ whose
endpoints $\rho(0)$ and $\rho(1)$ are contained in $\partial\D$.
An arc $\rho$ is called \emph{essential} if
it cannot be homotoped into $\partial \D$ with fixed endpoints. 
In the sequel we always assume that arcs are essential and
that
the number of intersections of $\rho$ with $\partial \D$ is minimal in
its isotopy class.

Choose an orientation of the curves in $\partial \D$. 
Since $\partial V$ is oriented, this choice determines a left and a
right side of a component $\alpha$ of $\partial \D$ in a small annular
neighborhood of $\alpha$ in $\partial V$. We then say that an endpoint
$\rho(0)$ (or $\rho(1)$) of an arc $\rho$ \emph{lies to the right (or
  to the left) of $\alpha$}, if a small neighborhood 
$\rho([0,\epsilon])$ (or $\rho([1-\epsilon,1])$) of this
endpoint is contained in the right (or left) side of $\alpha$ in a
small annulus around $\alpha$.
A \emph{returning arc relative to $\D$} is an arc both of whose
endpoints lie on the same side of the 
boundary $\partial D$ of a disk $D$
in $\D$, and whose interior is disjoint from $\partial \D$.

\medskip
Let $E$ be a disk which is not disjoint from $\D$. 
An \emph{outermost arc} of $\partial E$ relative to $\D$ is
a returning arc $\rho$ relative to $\D$, with endpoints on 
the boundary of a disc $D\in \D$, such that
there is a component $E^\prime$ of $E\setminus\D$ whose boundary is
composed of $\rho$ 
and an arc $\beta\subset D$. The interior of $\beta$
is contained in the interior of $D$. We call such a disk $E^\prime$
an \emph{outermost component} of $E\setminus\D$.

For every disk $E$ which is not disjoint from $\D$ there are 
at least two distinct outermost components $E^\prime,E^{\prime\prime}$
of $E\setminus\D$.
There may also be components of $\partial E\setminus\D$ which
are returning arcs, but not outermost arcs. 
For example, a component of $E\setminus \D$ may be a rectangle bounded by two 
arcs contained in $\D$ and two subarcs of $\partial E$ with endpoints
on $\partial \D$ which are homotopic to a returning arc relative to
$\partial \D$.   

Let now $\D$ be a simple disk system and let $\rho$ be a
returning arc whose endpoints are contained in the boundary of some
disk $D \in \D$. 
Then $\partial D \setminus \{\rho(0),\rho(1)\}$ is the union of
two (open) intervals $\gamma_1$ and $\gamma_2$. Put $\alpha_i =
\gamma_i\union\rho$. Up to isotopy, $\alpha_1$ and $\alpha_2$ are
simple closed curves in $\partial V$ 
which are disjoint from $\D$ (compare
\cite{St02} for this construction). Therefore
both $\alpha_1$ and $\alpha_2$ bound disks in the handlebody which we
denote by $Q_1$ and $Q_2$. We say that $Q_1$ and $Q_2$ are obtained
from $D$ by \emph{simple surgery along the returning arc $\rho$}.

The following observation is well known (compare \cite{HH11}, 
\cite[Lemma~3.2]{M86}, or \cite{St02}). 
\begin{lemma}\label{lemma:unique-reduced-exchange}
If $\Sigma$ is a reduced disk system and $\rho$ is a returning arc
with endpoints on $D \in \Sigma$, then for exactly one choice of the disks
$Q_1,Q_2$ defined as above, say the disk $Q_1$, the disk system
obtained from $\Sigma$ by replacing $D$ by $Q_1$ is reduced.
\end{lemma}
The disk $Q_1$ is characterized by the requirement
that the two spots in the boundary of $V\setminus\Sigma$ 
corresponding to the two copies of 
$D$ are contained in distinct connected components of
$H\setminus(\Sigma\cup Q_1)$. 
It only depends on $\Sigma$ and the
returning arc $\rho$.
We call the interval $\gamma_1$ used in the construction of the disk
$Q_1$ the \emph{preferred interval defined by the returning arc}.

\begin{defi}\label{diskmove}
Let $\Sigma$ be a reduced disk system. 
A \emph{disk exchange move} is the replacement
of a disk $D\in \Sigma$ by a disk $D^\prime$
which is disjoint from $\Sigma$ and 
such that $(\Sigma\setminus D)\cup D^\prime$ is a reduced
disk system. If $D^\prime$ is 
determined as in Lemma~\ref{lemma:unique-reduced-exchange} by a
returning arc of a disk in a disk system $\D$ then
the modification
is called a \emph{disk exchange move of $\Sigma$
in direction of $\D$} or simply a \emph{directed
disk exchange move}.

A sequence $(\Sigma_i)$ of reduced disk systems is called a \emph{disk
exchange sequence in direction of $\D$} (or \emph{directed disk
exchange sequence}) if each $\Sigma_{i+1}$ is obtained from $\Sigma_i$
by a disk exchange move in direction of $\D$.
\end{defi}

The following lemma is an easy consequence of the fact that simple
surgery reduces the geometric intersection number (see \cite{HH11} for
a proof).
\begin{lemma}\label{lemma:surgery-sequences}
  Let $\Sigma_1$ be a reduced disk system and let $\D$ be any other
  disk system.  
  Then there is a disk exchange sequence $\Sigma_1,\ldots,\Sigma_n$ in
  direction of $\D$ such that $\Sigma_n$ is disjoint from $\D$.
\end{lemma}

To estimate the growth rate of the Dehn function of the handlebody
group we will need to compare disk exchange sequences starting in
disjoint reduced disk systems. 
This is made possible by considering another type of surgery sequence for disk
systems, which we describe in the remainder of this section.

To this end, let $\D$ be any simple disk system and let $\rho$ 
be a returning arc. A \emph{full disk
replacement defined by $\rho$} modifies a simple disk system $\D$ to a
simple disk system $\D^\prime$ as follows. 
Let $D \in \D$ be the disk containing the endpoints of the returning
arc $\rho$.
Replace $D$ by both disks
$Q_1,Q_2$ obtained from $D$ by the simple surgery defined by $\rho$. 
The disks $Q_1,Q_2$ are disjoint from each other and from $\D$. If 
one (or both) of these disks is isotopic to a disk $Q$ from
$\D\setminus D$
then this disk will be discarded (i.e. we retain a single copy of $Q$;
compare \cite{Ha95} for a similar
construction). We say that a sequence $(\D_i)$ is a \emph{full disk
  replacement sequence in direction of $\D$} (or \emph{directed full
  disk replacement sequence}) if each $\D_{i+1}$ is obtained from
$\D_i$ by a full disk replacement along a returning arc contained in
$\partial \D$.

The following two lemmas relate full disk replacement sequences to disk
exchange sequences. Informally, these lemmas state that every directed
disk exchange sequence may be extended to a full disk replacement
sequence, and conversely every full disk replacement sequence contains a disk
exchange sequence. To make this idea precise, we use the following
\begin{defi}
  Let $\D$ be an arbitrary disk system.
  Suppose that $\D_0,\ldots,\D_n$ is a full disk replacement sequence in
  direction of $\D$ and that $\Sigma_1,\ldots,\Sigma_k$ is a disk
  exchange sequence in direction of $\D$.

  We say that the sequences $(\D_i)$ and $(\Sigma_i)$ are
  \emph{compatible}, if there is a non-decreasing surjective map
  $r:\{0,\dots,n\}\to \{1,\dots,k\}$ such that
  $\Sigma_{r(i)}\subset \D_{i}$ for all $i$.
\end{defi}
\begin{lemma}\label{lemma:fulldisk}
Let $\Sigma$ be a reduced disk system, let $\D$ 
be a simple disk system containing $\Sigma$ and let 
$\D=\D_0, \D_1,\dots,\D_m$ be
a full disk replacement sequence in direction of a disk system $\D'$. 
Then there is a disk exchange sequence $\Sigma=\Sigma_0,
\Sigma_1,\dots,\Sigma_u$ in direction of $\D'$ which is compatible
with $(D_i)$.
\end{lemma}
\begin{proof} 
We proceed by induction on the
length of the full disk replacement sequence $(\D_i)$.
If this length equals zero there is nothing to show. 
Assume that the claim holds true whenever this length
does not exceed $m-1$ for some $m>0$. 

Let $\D_0,\dots,\D_m$ be a full disk replacement
sequence of length $m$ and 
let $\Sigma\subset \D_0$ be a reduced disk system.
Let $D\in \D_0$ be the disk replaced in the full
disk replacement move connecting $\D_0$ to $\D_1$.

If $D\in \Sigma$ then for one of the two disks obtained 
from $D$ by simple surgery, say the disk $D^\prime$, the
disk system $\Sigma_1=(\Sigma\setminus D)\cup D^\prime$ is reduced.  
However, $\Sigma_1\subset \D_1$ 
and the claim now follows from the induction hypothesis.

If $D\not\in \Sigma$ then $\Sigma\subset \D_1$ by 
definition and once again, the claim follows from
the induction hypothesis.
\end{proof}

\begin{lemma}\label{typical}
Let $\Sigma_0,\dots,\Sigma_m$ be a disk exchange
sequence of reduced disk systems in direction of a disk system $\D'$. Then for
every simple disk system $\D_0\supset \Sigma_0$
there is a full disk replacement sequence 
$\D_0,\dots,\D_k$ in direction of $\D'$ which is compatible with $(\Sigma_i)$.
\end{lemma}
\begin{proof} We proceed by induction on the length $m$ of the
  directed disk exchange sequence.

The case $m=0$ is trivial, so assume that the lemma holds true
for directed disk exchange sequences of length at most $m-1$ for some $m\geq 1$.
Let $\Sigma_0,\dots,\Sigma_m$ be a directed disk exchange sequence of length
$m$.
Suppose $\Sigma_1$ is obtained from $\Sigma_0$ by replacing a disk
$D\in\Sigma_0$. Let $\rho$ be the returning arc with endpoints on $D$
defining the disk replacement, and let $D_1$ be the disk in $\Sigma_1$
which is the result of the simple surgery.

We distinguish two cases.
In the first case,
$\rho \cap \D_0 = \rho \cap D$.
Then $\rho$ is a returning arc relative to $\D_0$.
Let $\D_1$ be the disk system obtained from $\D_0$ by 
the full disk replacement defined by $\rho$.
One of the two disks obtained by simple surgery along $\rho$ 
is the disk $D_1$ and hence $D_1\in \D_1$.
The claim now follows from the induction hypothesis,
applied to the disk exchange sequence $\Sigma_1,\dots,\Sigma_m$ of
length $m-1$ and the simple disk system $\D_1$ containing $\Sigma_1$.

In the second case, the returning arc $\rho$ intersects $\D_0\setminus
D$. Then $\rho \setminus (\D_0\setminus D)$ contains a component $\rho'$ which
is a returning arc with endpoints on a disk $Q\in\D_0\setminus\{D\}$.
A replacement of the disk $Q$ by both disks obtained from
$Q$ by simple surgery using the returning arc $\rho'$
reduces the number of intersection 
components of $\rho \cap (\D_0\setminus D)$. Moreover,
the resulting disk system contains $D$. 
In finitely many surgery steps, say $s\geq 1$ steps, we
obtain a simple disk system $\D_s$ with the following properties.
\begin{enumerate}
\item $\D_s$ contains $D$ and is obtained from $\D_0$ by a 
full disk replacement sequence. 
\item $\rho \cap (\D_s-D)=\emptyset$.
\end{enumerate} 
Define $r(i)=0$ for $i=0,\dots,s$, where $r$ is the function required
in the definition of compatibility.
We now can use the procedure 
from the first case above, applied
to $\Sigma_0,\D_s$ and $\rho$ to 
carry out the induction step.

This completes the proof of the lemma.
\end{proof}

\section{The graph of rigid racks}
\label{sec:racks}
The goal of this section is to describe a construction of paths in the
handlebody group whose geometry is easy to control.
A version of these paths was already used in \cite{HH11} to establish
an upper bound for the distortion of the handlebody group in the
mapping class group. 

The main objects are given by the following
\begin{defi}\label{rack}
A \emph{rack} $R$ in $V$ is given by a reduced disk 
system $\Sigma(R)$, called the \emph{support system} of the rack $R$,
and a collection of pairwise disjoint essential
embedded arcs in $\partial V\setminus\partial\Sigma(R)$ 
with endpoints on $\partial\Sigma(R)$,
called \emph{ropes}, which are pairwise non-homotopic 
relative to $\partial \Sigma(R)$.
At each side of a support disk $D\in \Sigma(R)$, 
there is at least one rope which 
ends at the disk and approaches the disk from this side. A rack is
called \emph{large} if the set of ropes decomposes  $\partial
V\setminus\partial\Sigma(R)$ into simply connected regions.
\end{defi}

We will consider racks up to an equivalence relation called
``rigid isotopy'' which is defined as follows. 
\begin{defi}
  \begin{enumerate}[i)]
  \item Let $R$ be a large rack. The union of the support system 
    and the system of ropes of $R$ defines the $1$--skeleton of a cell
    decomposition of the surface $\partial V$ which we call the \emph{cell
      decomposition induced by $R$}.
  \item Let $R$ and $R'$ be racks. We say that $R$ and $R'$ are
    \emph{rigidly isotopic} if 
    there is an isotopy of $\partial V$ which maps the support system
    of $R$ to the support system of $R^\prime$ and defines an isotopy
   of the cell decompositions induced by $R$
    and $R'$.
  \end{enumerate}
\end{defi}
In particular, if $T$ is a simple Dehn twist about the boundary
of a support disk of a rack $R$, then $R$ and $T^n(R)$ are not
rigidly isotopic for $n\geq 2$. This observation and the fact that the
stabilizer in the mapping class group of a reduced disk system is
contained in the handlebody group imply the following
\begin{lemma}\label{lemma:marked-finite-quotient}
  The handlebody group acts on the set of rigid isotopy classes of
  racks with finite quotient and finite point stabilizers.
\end{lemma}
For simplicity of notation, we call a 
rigid isotopy class of a large rack simply a \emph{rigid rack}. 
Lemma~\ref{lemma:marked-finite-quotient} allows us to use rigid racks
as the vertex set of a $\Map(V)$--graph. More precisely, we make the following
\begin{defi}
  The \emph{graph of rigid racks} $\mathcal{RR}_K(V)$ is the graph whose
  vertex set is the set of rigid racks. Two such
  vertices are joined by an edge if up to isotopy, 
  the $1$--skeleta of the cell decompositions induced by the
  racks intersect in at most $K$ points.
\end{defi}
It follows easily from Lemma \ref{lemma:marked-finite-quotient}
that the number $K$ may be chosen in such a way that the graph
$\mathcal{RR}_K(V)$ is
connected. In Lemma~7.3 of \cite{HH11} such a number $K>0$
is constructed explicitly.
In the sequel, we will always use this choice of $K$ and suppress the
mention of $K$ from our notation.
It then follows from 
Lemma~\ref{lemma:marked-finite-quotient} 
and the Svarc-Milnor lemma that 
the graph $\mathcal{RR}(V)$ is quasi-isometric to
$\Map(V)$.  

\medskip
Next we construct a family of distinguished paths 
in the graph of rigid racks. The paths are inspired by
splitting sequences of train tracks on surfaces. To
this end, we first define a notion of ``carrying'' for racks.
\begin{defi}\label{carrying}
\begin{enumerate}
\item A disk system $\D$  
is \emph{carried} by a rigid rack $R$ if it is in minimal
position with respect to the support system 
$\Sigma(R)$ of $R$ and if each component of 
$\partial \D\setminus\partial\Sigma(R)$
is homotopic relative to $\partial \Sigma(R)$ 
to a rope of $R$.
\item An embedded essential arc $\rho$ in $\partial V$ 
with endpoints on $\partial \Sigma(R)$ 
is \emph{carried} by $R$ if each component of 
$\rho\setminus\partial\Sigma(R)$ is homotopic relative to
$\partial\Sigma(R)$ to a rope of $R$. 
\item A \emph{returning rope} of a rigid rack $R$ is a rope which
begins and ends at the same side of some fixed 
support disk $D$ (i.e. defines a returning arc relative to $\partial
\Sigma(R)$). 
\end{enumerate}
\end{defi}
Let $R$ be a rigid rack with support system $\Sigma(R)$ and let $\alpha$ be a 
returning rope of $R$ with endpoints on a support disk $D\in \Sigma(R)$.
By Lemma~\ref{lemma:unique-reduced-exchange}, for one
of the components $\gamma_1,\gamma_2$ of 
$\partial D\setminus\alpha$, say the component $\gamma_1$, 
the simple closed curve
$\alpha\cup \gamma_1$ is the 
boundary of an embedded disk $D^\prime\subset H$
with the property that 
the disk system
$(\Sigma\setminus D)\cup D^\prime$ is reduced.

A \emph{split} of the rigid rack $R$
at the returning rope $\alpha$ is any rack $R^\prime$
with support system $\Sigma^\prime=(\Sigma(R)\setminus D)\cup D^\prime$
whose ropes are given as follows.
\begin{enumerate}
\item Up to isotopy, each 
rope $\rho^\prime$ of $R^\prime$ has its endpoints in 
$(\partial\Sigma(R)\setminus\partial D)\cup
\gamma_1\subset \partial\Sigma(R)$ and  
is an arc carried by $R$. 
\item For every rope $\rho$ of $R$ there is a rope $\rho^\prime$
of $R^\prime$ such that up to isotopy, $\rho$ is a component of 
$\rho^\prime\setminus\partial\Sigma(R)$.
\end{enumerate}

The above definition implies in particular that 
a rope of $R$ which does not have an endpoint on $\partial D$ is 
also a rope of $R^\prime$. Moreover, there is a map
$\Phi:R^\prime\to R$ which maps a rope of 
$R^\prime$ to an arc carried by $R$, and which maps 
the boundary of a support disk of $R^\prime$ to a simple closed
curve $\gamma$ of the form $\gamma_1\circ \gamma_2$ where
$\gamma_1$ either is a rope of $R$ or trivial, 
and where $\gamma_2$ is
a subarc of the boundary of a support disk of $R$
(which may be the entire boundary circle).
The image of $\Phi$ contains every rope of $R$. 

\medskip
We are now ready to recall the construction of a distinguished class
of edge-paths in the graph of rigid racks from \cite{HH11}. These paths are
sufficiently well-behaved to yield 
some geometric control of the handlebody group.

For a reduced disk system $\Sigma$ let $\mathcal{RR}(V,\Sigma)$
be the complete subgraph of $\mathcal{RR}(V)$ whose vertices are marked
rigid racks with support system $\Sigma$.

\begin{defi}
  Let $\D$ be a simple disk system. A \emph{$\D$-splitting sequence} of
    racks is an edge-path $R_i$ in the graph of rigid racks with the
  following properties.
  \begin{enumerate}[i)]
  \item There is a disk exchange sequence $\Sigma_i$ in direction of $\D$
    and a sequence of numbers $1=r_1<\dots<r_k$ such that the support
    system of $R_j$ is $\Sigma_i$ for all $r_i\leq j\leq
    r_{i+1}-1$. The sequence $\Sigma_i$ is called the \emph{associated
    disk exchange sequence}.
  \item For $r_i\leq j\leq r_{i+1}-1$, the sequence $R_j$  is a uniform
    quasi-geodesic in the graph $\mathcal{RR}(V,\Sigma_i)$.
  \end{enumerate}
\end{defi}
Here and in the sequel, we say that a path is a \emph{uniform}
quasigeodesic if the quasigeodesic constants of the path depend only
on the genus of the handlebody. Similarly, we say that a number is
\emph{uniformly bounded}, if there is a bound depending only on the
genus of $V$.

We showed in \cite{HH11} that any two points in the graph of rigid
racks can be connected by a splitting sequence. More precisely,
the proof of Theorem~7.9 of \cite{HH11} yields
\begin{theorem}\label{thm:connect}
  Let $R,R'$ be two rigid racks. 
  Then there is a disk system ${\mathcal D}$ depending only on the support
  system of $R'$ with the following property.
  Let $\Sigma(R)=\Sigma_1,\Sigma_2,\ldots,\Sigma_n$ be a disk
  exchange sequence in direction of $\D$ such that $\Sigma_n$
  is disjoint from $\D$.

  Then there is a splitting sequence connecting $R$ to $R'$ whose
  associated disk exchange sequence is $(\Sigma_i)$. The length of
  such a sequence is bounded uniformly 
  exponentially in the distance between $R$ and
  $R'$ in the graph of rigid racks.
\end{theorem}

In Section~\ref{sec:disk-systems} we saw that $\D$--disk exchange sequences
starting in disjoint reduced disk systems can be compared using full
disk replacement sequences. In the rest of this section we develop a
slight generalization of racks, which will allow to similarly compare
$\D$--splitting sequences starting in adjacent vertices of $\mathcal{RR}(V)$.

Namely, define an \emph{extended rack} $R$ in the same way as a rack
except that now the support system $\D(R)$ of $R$ may be any simple
disk system instead of a reduced disk system.
The cell decomposition induced by an extended rack is defined in the
obvious way, and similarly we can talk about rigid isotopies
between extended racks. The rigid isotopy class of an
extended rack is called a \emph{rigid extended rack}. 

Rigid extended racks can be used in the same way as racks to define a
geometric model for the handlebody group. 
\begin{defi}
  The graph of rigid extended racks $\mathcal{RER}_K(V)$ is the graph
  whose vertex set is 
  the set of large rigid extended racks. Two such vertices are connected by
  an edge of length one if up to isotopy the $1$--skeleta of the cell
  decompositions induced by the corresponding vertices intersect in
  at most $K$ points. 
\end{defi}
Again, the constant $K$ is chosen in such a way that the graph of
rigid extended racks is connected. We denote the resulting graph by
$\mathcal{RER}(V)$. For future use, we choose the constant $K$ big enough
such that in addition the following holds.
For a simple disk system $\D$ let $\mathcal{RER}(V,\D)$
be the complete subgraph of $\mathcal{RER}(V)$ whose vertices are 
rigid extended racks with support system $\D$. We may choose $K$ large
enough such that for any
simple disk system $\D$ the subgraph  $\mathcal{RER}(V,\D)$ is connected.

An analog of Lemma~\ref{lemma:marked-finite-quotient} holds for rigid
extended racks as well, and implies that the handlebody group acts on
the graph of rigid extended racks with finite quotient.
Thus, the graph of rigid extended racks is quasi-isometric to the
handlebody group. Note also that every large rack is a large extended
rack. Thus the graph of rigid extended racks embeds as a
subgraph in the graph of rigid extended racks. This inclusion is a
quasi-isometry. 

\medskip
A \emph{full split} of a rigid extended rack is defined as follows. Let $R$
be a rigid extended rack and let  $\alpha$ be a returning rope of $R$. A
rigid extended rack $R'$ is called a full split of $R$ at $\alpha$
if the support system of $R'$ is obtained from $\Sigma(R)$ by a
full disk replacement along $\alpha$. Moreover,
we require that the ropes of $R'$
satisfy the analogous conditions as the ropes of a split of a
rigid rack. 

The following is a natural generalization of splitting paths to
extended racks.
\begin{defi}
  Let $\D$ be a simple disk system. A \emph{full $\D$-splitting sequence} of
  racks is an edge-path $(R_i)$ in the graph of rigid extended racks with the
  following properties.
  \begin{enumerate}[i)]
  \item There is a full disk exchange sequence $(\D_i)$ in direction of $\D$
    and a sequence of numbers $1=r_1<\dots<r_k$ such that the support
    system of $R_j$ is $\D_i$ for all $r_i\leq j\leq
    r_{i+1}-1$. The sequence $(\D_i)$ is called the \emph{associated
    full disk exchange sequence}.
  \item For $r_i\leq j\leq r_{i+1}-1$, the sequence $(R_j)$  is a uniform
    quasi-geodesic in the graph $\mathcal{RER}(V,\Sigma_i)$.
  \end{enumerate}
\end{defi}
The proof of Theorem~7.9 of \cite{HH11}
implies the following theorem 
which allows to connect two rigid racks with a full splitting sequence.
\begin{theorem}\label{thm:connect-extended}
  There is a number $k_1$ with the following property. 
  Let $R,R'$ be two rigid racks. 
  Then there is a simple disk system $\hat{\D}$ depending only on the
  support system of $R'$ with the following property.
  Let $\D(R)=\D_1, \D_2, \ldots, \D_n$ be a full disk
  exchange sequence in direction of $\hat{\D}$ such that $\D_n$
  is disjoint from $\hat{\D}$.

  Then there is an extended rigid rack $\hat R$ which is at distance at
  most $k_1$ to $R'$ in $\mathcal{RER}(V)$, and there is 
  a full splitting sequence connecting $R$ to $\hat{R}$ whose
  associated full disk replacement sequence is $(\D_i)$.
  The length of any such sequence is bounded by $e^{k_1d}$, where $d$ is
  the distance between $R$ and $R'$ in the graph of rigid extended racks.
\end{theorem}

Combining Lemmas~\ref{lemma:fulldisk} and~\ref{typical} with
Theorem~\ref{thm:connect-extended} above, we obtain the following.
\begin{cor}\label{cor:extend-or-pick-racks}
  There is a number $k_2>0$ with the following property.
  \begin{enumerate}[i)]
  \item Let $(R_i), i=1,\ldots N$ be a $\D$--splitting sequence of racks with
    associated disk exchange sequence $(\Sigma_j)$. Let $(\D_j)$ be a
    full disk replacement sequence compatible with $(\Sigma_j)$. Then there is a
    full $\D$--splitting sequence $\widetilde{R}_k, k=1,\ldots K$ such that the
    following holds. The associated full disk replacement sequence to
    $(\widetilde{R}_k)$ is $\D_j$. 
    Furthermore, $\widetilde{R}_1 = R_1$ and the distance between
    $\widetilde{R}_K$ and $R_N$ is at most $k_1$. The length $K$ of
    any such sequence is at most $e^{k_2d}$, where $d$ is the distance
    between $R_0$ and $R_K$ in the graph of rigid racks.
  \item Conversely, suppose that $\widetilde{R}_k, k=1,\ldots, K$ is a full
    $\D$--splitting 
    sequence with associated full disk replacement sequence
    $\D_j$. Suppose further that $(\Sigma_j)$ is a disk exchange sequence
    compatible with $(\D_i)$. 
    If $\widetilde{R}_1$ is a large rack, then there is a
    $\D$--splitting sequence 
    $R_1,R_2,\ldots, R_N$ whose associated disk exchange sequence is
    $(\Sigma_j)$ such that $R_N$ is of distance at most $k_1$ to
    $\widetilde{R}_K$. The length $N$ of
    any such sequence is at most $e^{k_2d}$, where $d$ is the distance
    between $R_0$ and $R_N$ in the graph of rigid racks.
  \end{enumerate}
\end{cor}

\section{The Dehn function of the handlebody group}
\label{s:upperbound-and-dehn-function}
In this section we prove the main result of this note.
\begin{theorem}\label{thm:main}
  The Dehn function of the handlebody group has at most exponential
  growth rate.
\end{theorem}
To begin, we recall the definitions of the Dehn function and growth rate.
Let $G$ be a finitely presented group. Choose a finite
generating $\mathcal{S}$ and let $\mathcal{R}$ be a 
finite defining set of
relations for $G$. This means the following.
The set $\mathcal{R}$ generates a subgroup $R_0$
of the free group $F(\mathcal{S})$ with generating set
$\mathcal{S}$. 
Denote by $R$ the normal closure of
$R_0$ in $F(\mathcal{S})$. The set $\mathcal{R}$ is called a defining
set of relations for $G$ if 
the quotient $F(\mathcal{S}) / R$ is isomorphic to $G$.

Every $r\in R<F({\mathcal S})$ can be written as a product of
conjugates of elements in $\mathcal{R}$:
$$r = \Pi_{i=1}^n r_i^{\gamma_i}, \quad\quad
r_i\in\mathcal{R},\gamma_i\in G.$$
We call the minimal length $n$ of such a product \emph{the area
  $\mathrm{Area}(r)$ needed to fill the relation $r$}. On the other
hand, $r$ can be written 
as a word in the elements of $\mathcal{S}$. We call the minimal
length of such a word the \emph{the length $l(r)$ of the loop $r$}.

The \emph{Dehn function of $G$} is then be defined by
$$\delta(n) = \max\{\mathrm{Area}(r) | r\in R \mbox{ with } l(r)\leq n\}.$$

The function $\delta$ depends on the choice of the generating set
$\mathcal{S}$ and the set of relations $\mathcal{R}$. However, 
the Dehn function obtained from
different generating sets and defining relations
are equivalent in the following sense. Say
that two functions $f, g: \NN \to \NN$ are \emph{of the same growth
  type}, if there are numbers $K,L>0$ such that 
$$L^{-1} \cdot g(K^{-1}\cdot x - K) - L \leq f(x) \leq L \cdot
g(K\cdot x + K) + L$$ 
for all $x \in \NN$. 

In this section we use the graph $\mathcal{RER}(V)$ of rigid extended racks as a
geometric model for the handlebody group.

To estimate the Dehn function, we consider a loop $\gamma$ in
$\mathcal{RER}(V)$ of length $R>0$. We have to show that there
is a number $k>0$ and that there are at most $e^{kR}$ 
loops $\zeta_1,\dots,\zeta_m$ 
of length at most $k$ so that $\gamma$ can
be contracted to a point in $m$ steps consisting each
of replacing a subsegment of $\zeta_i$ by another
subsegment of $\zeta_i$. This suffices, since each loop $\zeta_i$ as
above corresponds to a cycle in the handlebody group which can be
filled with uniformly small area.

\medskip
Recall from Section~\ref{sec:racks} the definition of the graph
$\mathcal{RER}(V,\D)$. The following lemma allows to control the
isoperimetric function of these subgraphs.
\begin{lemma}\label{lemma:filling-in-map}
  Let $\D$ be a simple disk system for $V$.
  \begin{enumerate}[i)]
  \item $\mathcal{RER}(V,\D)$ is a connected subgraph of
    $\mathcal{RER}(V)$ which is equivariantly quasi-isometric to the stabilizer
    of $\partial \D$ in the mapping class group of $\partial V$.
  \item $\mathcal{RER}(V,\Sigma)$ is quasi-isometrically embedded in
    $\mathcal{RER}(V)$.
  \item Any loop in $\mathcal{RER}(V,\Sigma)$ can be filled with area
    coarsely bounded quadratically in its length.
  \end{enumerate}
\end{lemma}
\begin{proof}
  $\mathcal{RER}(V,\Sigma)$ is connected by definition of the graph of
  rigid extended racks (see Section~\ref{sec:racks}).
  
  Let $G$ be the stabilizer of $\partial \D$ in the mapping class group
  of $\partial V$. The group $G$ is contained in the handlebody group 
  since every homeomorphism of the boundary of a spotted ball extends
  to the interior. The group $G$ acts on
  $\mathcal{RER}(V,\D)$ with finite quotient and finite point
  stabilizers. 
  To show this, note that up to the action of the
  mapping class group, there are only
  finitely many isotopy classes of cell decompositions 
  of a bordered sphere with uniformly few cells. 
  Thus by the Svarc-Milnor lemma, 
  $\mathcal{RER}(V,\D)$ is equivariantly quasi-isometric to $G$, showing
  \textit{i)}.

  The stabilizer $G$ of $\partial \D$ 
  is quasi-isometrically embedded in the full
  mapping class group of $\partial V$ (see \cite{MM00} or
  \cite[Theorem~2]{H09b}).
  Hence 
   $G$ is also quasi-isometrically embedded 
   in the handlebody
  group. Together with \textit{i)} this shows \textit{ii)}.

  The group $G$ is 
  a Lipschitz retract of the mapping class group 
 of $\partial V$
  (see \cite{HM10} for a detailed discussion of this fact which
  is a direct consequence of the work of Masur and Minsky
  \cite{MM00}).
   Mapping class groups are automatic \cite{Mo95} and
  hence have quadratic Dehn function. Then the same holds
  true for $G$ (compare again \cite{HM10}).
  This implies claim \textit{iii)}.
\end{proof}

\medskip

As the next step, we use Corollary~\ref{cor:extend-or-pick-racks}
to control splitting paths starting at adjacent
points in the graph of marked racks. We show that 
these paths can be constructed in such a way that the
resulting loop can be filled with controlled area.
Together with the length estimate for marked splitting paths from
Theorem~\ref{thm:connect-extended} this will imply the exponential 
bound for the Dehn function. 

The main technical tool in this approach is given by the following lemma.
\begin{lemma}\label{lem:compare-adjacent}
  For each $k>0$ there is a number $k_3>0$ with the following property.

  Let $\D$ be a simple disk system.
  Let $R_i, i=1,\ldots, N$ be a $\D$--splitting sequence of rigid
  racks and let $\widetilde{R}_j, j=1,\ldots, M$ be a full
  $\D$--splitting sequence of extended racks such that the following
  holds.
  \begin{enumerate}[i)]
  \item The rigid extended racks $R_1$ and $\widetilde{R}_1$ (respectively
    $R_N$ and $\widetilde{R}_M$) have distance at most $k$ in the
    graph of rigid extended racks.
  \item The associated disk exchange sequences of $R_i$ and
    $\widetilde{R}_j$ are compatible.
  \end{enumerate}
  Then the loop $\gamma$ in $\mathcal{RER}(V)$ formed by the sequences
  $(R_i)$, $(\widetilde{R}_j)$ and geodesics between $R_1$ and
  $\widetilde{R}_1$ and $R_N$ and $\widetilde{R}_M$ can be filled with
  area $k_3(N+M)^3$.
\end{lemma}
\begin{proof}
  The idea of the proof is to inductively decompose the loop $\gamma$ into
  smaller loops, each of which can be filled with area at most
  $k_3(N+M)^2$ for a suitable $k_3$.

  Denote the disk exchange sequence associated to $R_i$ by $\Sigma_i, i=1,\dots,n)$ 
  and the full disk replacement sequence associated to $\widetilde{R}_j$
  by $\D_j, j=1,\ldots m$.
  Let $r:\{1,\ldots,m\}\to\{1,\ldots,n\}$ be the monotone
  non-decreasing surjective
  function given by compatibility, i.e. $\Sigma_{r(j)} \subset \D_j$
  for all $j=1,\ldots, m$.

  We define
  $$I(i) = \{ k\mid \Sigma(R_k) = \Sigma_i \}$$
  and
  $$J(i) = \{ k\mid \D(\widetilde{R}_k)=\D_l \mbox{ and } r(l)=i \}.$$
  Put $i_k = \max I(k)$ and $j_k = \max J(k)$. We will inductively
  choose paths $d_k$ connecting $R_{i_{k}}$ to $\widetilde{R}_{j_k}$
  and paths $c_k$ connecting $R_{i_k + 1}$ to $\widetilde{R}_{j_k +
  1}$ with the following properties.
  \begin{enumerate}[i)]
     \item The path $c_k$ is a uniform quasigeodesic in
       $\mathcal{RER}(V,\Sigma_{k+1})$. 
     \item The path $d_k$ is a uniform quasigeodesic in
       $\mathcal{RER}(V,\Sigma_k)$. 
     \item The paths $c_{k+1}, d_k$ are uniform fellow-travelers, i.e. the
     Hausdorff distance between $c_{k+1}$ and $d_k$ is uniformly bounded.
  \end{enumerate}
  A family of paths with these properties implies the statement of the
  lemma in the following way.

  The restriction of the sequence $R_i$ to $I(k)$ and the restriction
  of $\widetilde{R}_j^{-1}$ to $J(k)$ form together with $c_{k-1}^{-1}$ and $d_k$ a
  loop $\gamma_k$ in $\mathcal{RER}(V,\Sigma_k)$. 
  The length of $c_{k-1}$ and $d_k$ is coarsely bounded by $N+M$ by
  the triangle inequality. Hence, the length of $\gamma_k$ can be
  coarsely bounded by $4(N+M)$. 
  Since  $\mathcal{RER}(V,\Sigma_k)$ admits a
  quadratic isoperimetric function, this loop can thus be filled with area
  bounded by $k_3(N+M)^2$ for some uniform constant $k_3$.

  Similarly, the paths $d_k^{-1}$ and $c_{k+1}$ , together with the edges connecting
  $R_{i_k}$ to $R_{i_k+1}$ and  $\widetilde{R}_{j_k}$ to
  $\widetilde{R}_{j_k+1}$ form a loop $\delta_k$. The length of
  $\delta_k$ can again be coarsely
  bounded by $2(N+M)$ using the triangle inequality. Since the paths
  $d_k$ and $c_k$ are fellow-travelers, $\delta_k$
  can be filled with area depending linearly on its length. 

  There are at most $2\max(N,M)$ loops
  $\gamma_k,\delta_k$. Hence, the concatenation of all the loops
  $\gamma_i$ and $\delta_j$ can be filled with area at most
  $k_3(N+M)^3$ (after possibly enlarging the constant $k_2$).
  The paths $c_i$ and $d_i$ occur
  in the concatenation of $\gamma_i$ and $\delta_j$ twice, with
  opposite orientations, except for $c_0$ and the last occurring arc $d_L$.
  As a consequence, the concatenation of the loops $\gamma_i$ and
  $\delta_j$ is, after erasing these opposite paths, uniformly
  close to $\gamma$ in the Hausdorff metric. Thus, $\gamma$ may also
  be filled with area bounded by $k_3(N+M)^2$ (again possibly
  increasing $k_3$).

  \medskip
  We now describe the inductive construction of the paths $c_k$ and
  $d_k$. We set $c_0=d_0$ to be the constant path $R_1$.
  Suppose that the paths $c_i, d_i$ are already constructed for
  $i=0,\ldots, k-1$. 

  The support systems of $R_{i_k}$ and $\widetilde{R}_{j_k}$ both contain
  $\Sigma_k$. We first construct the path $d_{k}$
  connecting $R_{i_k}$ and $\widetilde{R}_{j_k}$.
  
  Namely, the reduced disk systems $\Sigma_k$ and $\Sigma_{k+1}$ are
  disjoint. The 
  simple disk system $\Sigma_k\cup \Sigma_{k+1}$ is disjoint from the
  support systems of $R_{i_k}, R_{i_k+1}$ and
  $\widetilde{R}_{j_k}, \widetilde{R}_{j_k+1}$ by definition of a
  split. Furthermore, the $1$-skeleta of the cell decompositions of
  all four of these extended racks intersect
  $\partial\Sigma_k\cup\partial\Sigma_{k+1}$ in uniformly few points.
  Hence, there are rigid extended racks $U_1, U_2$ which have
  $\Sigma_k\cup\Sigma_{k+1}$ as their support system and such that $U_1$
  is uniformly close to $R_{i_k}$, and $U_2$ is uniformly close to
  $\widetilde{R}_{j_k}$ in $\mathcal{RER}(V)$. Let $e$ be a geodesic path in
  $\mathcal{RER}(V,\Sigma_k\cup\Sigma_{k+1})$ connecting $U_1$ and
  $U_2$. 
  Since  $\mathcal{RER}(V,\Sigma_k\cup\Sigma_{k+1})$ is undistorted in
  $\mathcal{RER}(V)$ by Lemma~\ref{lemma:filling-in-map}, the length of $e$ 
  is coarsely bounded by $N+M+1$.
  By adding uniformly short geodesic segments in $\mathcal{RER}(V,\Sigma_k)$
  at the beginning and the end of $e$, we obtain the path $d_{k}$ with
  property \textit{ii)}.

  By definition of $i_k$ and $j_k$, we have $i_k+1 \in I(k+1)$ and
  $j_k+1 \in J(k+1)$. Hence, both $R_{i_k+1}$ and
  $\widetilde{R}_{j_k+1}$ contain $\Sigma_{k+1}$ in their support
  systems. 
  We can thus define $c_{k}$ with properties \textit{i)} and
  \textit{iii)} by adding 
  uniformly short geodesic segments in $\mathcal{RER}(V,\Sigma_{k+1})$
  to the beginning and the end of $e$. 
\end{proof}
We have now collected all the tools for the proof of the main theorem.
\begin{proof}[Proof of Theorem~\ref{thm:main}]
  Recall that it suffices to show that every loop in the graph of rigid
  racks can be filled with area coarsely bounded by an exponential
  function of its length. 

  Let $R_i$ be a loop of length $L$ in the graph of rigid
  racks based at $R_0=\hat{R}$. 
  Let $\hat{\Sigma}$ be the disk system given by
  Theorem~\ref{thm:connect-extended} applied to $R' = R_0$. 
  Since the graph of rigid racks is
  quasi-isometric to the graph of extended rigid racks, we can
  consider $R_i$ as a loop in $\mathcal{RER}(V)$ and it suffices to
  show that this loop can be filled in $\mathcal{RER}(V)$ with area
  bounded exponentially in its length.

  The strategy of this proof is similar to the proof of
  Lemma~\ref{lem:compare-adjacent}: we will write the loop $(R_i)$ as
  a concatenation of smaller loops whose area we can control.  

  We will define paths $c_i$ in $\mathcal{RER}(V)$ with the following
  properties. 
  \begin{enumerate}
  \item The path $c_i$ connects $R_i$ to a rack which is uniformly
    close to $\hat{R}$ in $\mathcal{RER}(V)$.
  \item The path $c_i$ is a $\hat{\Sigma}$--splitting sequence of racks.
  \item The loop formed by $c_i, c_{i+1}$, the edge between $R_i$ and
    $R_{i+1}$, and a geodesic connecting other pair of endpoints of
    $c_i, c_{i+1}$ can be filled with area bounded by $e^{k_3L}$.
  \end{enumerate}
  As a consequence, the loop $(R_i)$ itself can be filled with area at
  most $Le^{k_3L}$, proving the theorem.

  \medskip
  The construction of the paths $c_i$ is again by induction.
  We set $c_0$ to be the constant path $R_0$. Suppose now that the
  path $c_k$ is already constructed. 
  Since $R_k$ and $R_{k+1}$ are connected by an edge in the graph of
  rigid racks, their support systems $\Sigma_k$ and $\Sigma_{k+1}$ are
  disjoint. Let $\Sigma_k^{(i)}, i=1,\ldots,n$ be the disk exchange
  sequence associated to the splitting sequence $c_k$. 
  Put $\D_1 = \Sigma_k \cup \Sigma_{k+1}$. Using Lemma~\ref{typical}
  we obtain a full disk replacement sequence $(\D_i)$ compatible with
  $(\Sigma_k^{(i)})$. Corollary~\ref{cor:extend-or-pick-racks} part
  \textit{i)} then yields a full splitting sequence
  $\widetilde{R}_k,k=1,\ldots,M$ with associated full disk exchange
  sequence $(\D_i)$. By Lemma~\ref{lemma:filling-in-map}, the loop
  formed by $c_k$ and $\widetilde{R}_k,k=1,\ldots,M$ can be filled
  with area bounded by $k_2(N+M)^3$, where $N$ is the length of the
  path $c_k$.
  
  Using Lemma~\ref{lemma:fulldisk} on the sequence $(\D_i)$ and the
  initial reduced disk system $(\Sigma_{k+1})$ we obtain a
  $\hat{\Sigma}$--splitting sequence $(\Sigma_{k+1}^{(i)})$ compatible
  with $(\D_i)$, which starts in $\Sigma_{k+1}$.
  Corollary~\ref{cor:extend-or-pick-racks} part \textit{ii)} now
  yields a $\hat{\Sigma}$--splitting sequence $c_{k+1}$ starting in
  $R_{k+1}$ and ending uniformly close to $\hat{R}$.
  Applying Lemma~\ref{lem:compare-adjacent} again, we see that the
  loop formed by $c_{k+1}$ and $\widetilde{R}_k,k=1,\ldots,M$ can be filled
  with area bounded by $k_3(N'+M)^3$, where $N'$ is the length of the
  path $c_{k+1}$.

  Since both $c_k$ and $c_{k+1}$ are splitting sequences connecting
  points which are of distance at most $L$, their lengths can be
  bounded by $k_4e^{k_4L}$ for a suitable $k_4$ by
  Theorem~\ref{thm:connect}. As a consequence, the paths $c_k$ and
  $c_{k+1}$ satisfy condition \textit{iii)}.
  This concludes the inductive construction of $c_i$ and the proof of
  the theorem.
\end{proof}

The proof of the theorem would give a polynomial bound for the
Dehn function provided that the length of the splitting paths
used to fill in loops had a length which is polynomial in the
distance between their endpoints. Unfortunately, however,
the following example show that such a bound does not exist.
This is similar to the behavior of paths of sphere systems
used in \cite{HV96} to show an exponential upper bound for the
Dehn function of ${\rm Out}(F_n)$, 

For simplicity of exposition, we do not construct these paths in the
graph of rigid racks (or the handlebody group), but instead in a
slightly simpler graph. The example given below can be extended to the
full graph of rigid racks in a straightforward fashion.

We define $\mathcal{RD}(V)$ to be the graph of reduced disk systems in
$V$. The vertex 
set of $\mathcal{RD}(V)$ is the set of isotopy classes of reduced
disk systems, and two such vertices are connected by an edge of length
one if the corresponding disk systems are disjoint. Every directed
disk exchange sequence defines an edge-path in
$\mathcal{RD}(V)$. The following example shows that
the length of  these edge-paths may be exponential in the distance
between their endpoints.
\begin{example}
  Consider a handlebody $V$ of genus $4$. 
  For each $n \in \NN$ we will construct a disk exchange sequence
  $\Sigma^{(n)}_1,\ldots, \Sigma^{(n)}_{N(n)}$ such that on the one
  hand, the length $N(n)$ of the sequence growth exponentially in
  $n$. On the other hand, the distance between endpoints $\Sigma^{(n)}_1$ and
  $\Sigma^{(n)}_{N(n)}$ in $\mathcal{RD}(V)$ grows linearly in $n$.
  To simplify the notation, in this example we will only construct
  the endpoint $\Sigma^{(n)}_{N(n)}$ and denote it by $\Sigma_n$.

  We choose three disjoint
  simple closed curves $\alpha_1,\alpha_2,\alpha_3$ which 
  decompose the surface $\partial V$ into a pair of pants, two
  once-punctured tori and a once-punctured genus $2$ surface (see
  Figure~\ref{fig:example}). 
  We may choose the $\alpha_i$ such that they bound disks in $V$. We
  denote the two solid tori in the complement of these disks by
  $T_1,T_2$ and the genus $2$ subhandlebody by $V'$.

  \begin{figure}[h!]
    \centering
    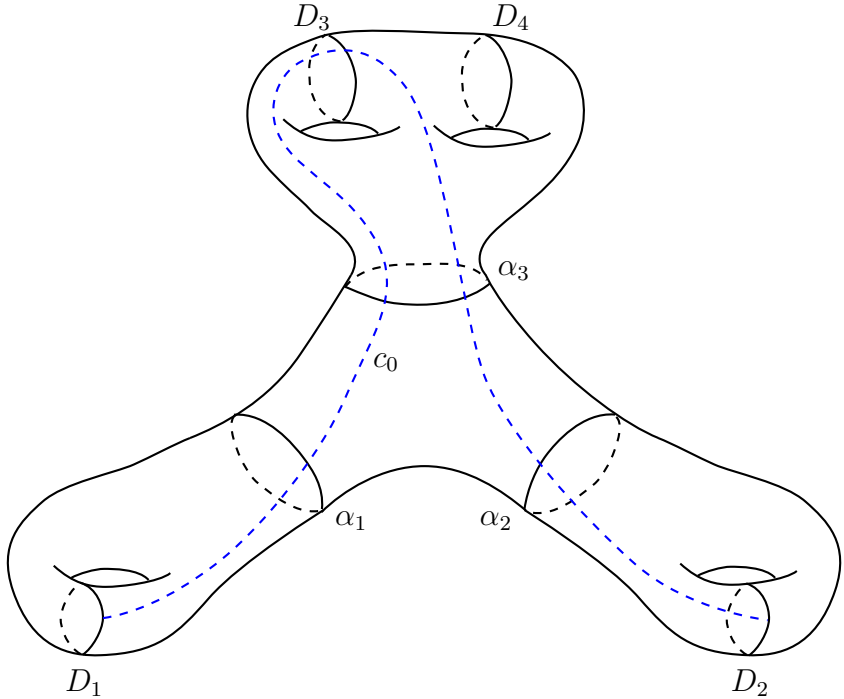
    \caption{The setup for the example of a non-optimal disk exchange
      path. An admissible arc is drawn dashed.}
    \label{fig:example}
  \end{figure}

  Let $\Sigma_0=\{D_1,D_2,D_3,D_4\}$ be a reduced disk system such
  that $D_1 \subset T_1, D_2 \subset T_2$ and $D_3,D_4 \subset V'$.
  Choose a base point $p$ on
  $\alpha_3$. Let $\gamma_1,\gamma_2$ be two disjointly embedded loops
  on $\partial V \cap V'$ based at $p$ with the following property.
  The loop $\gamma_1$ intersects the disk $D_3$ in a single point and
  is disjoint from $D_4$, while $\gamma_2$ intersects $D_4$ in a
  single point and is disjoint from $D_3$.
  Since the complement of $D_3 \cup D_4$ in $V'$ is simply connected,
  such a pair of loops generates the 
  fundamental group of $V'$. Denote the projections of $\gamma_1$ and
  $\gamma_2$ to $\pi_1(V',p)$ by $A_1$ and $A_2$, respectively.

  Let $c$ be an embedded arc on $\partial V$.
  We say that $c$ is \emph{admissible} if the following holds.
  The arc $c$ connects the disk $D_1$ to the disk
  $D_2$. The interior of $c$ intersects $\alpha_1$ and $\alpha_2$ in a single
  point each. Furthermore it intersects $\alpha_3$ in two points,
   and its interior is 
  disjoint from both $D_1$ and $D_2$. 

  Let $c$ be an admissible arc. The intersection of $c$ with $V'$
  is an embedded arc $c'$ connecting $\alpha_3$ to itself. The arc $c'$
  may be turned into an embedded arc in $V'$ based at $p$ by
  connecting the two endpoints of $c'$ to $p$ along $\alpha_3$. Since the
  curve $\alpha_3$ bounds a disk in $V'$, the image of this loop in 
  $\pi_1(V',p)$ is determined by the homotopy class of the arc $c$ relative
  to $\partial D_1,\partial D_2$.
  We call this image the \emph{element induced by the arc $c$}.

  Choose an admissible arc $c_0$ in 
  such a way that it intersects the disk $D_3$ in a single point, and
  is disjoint from $D_4$ (see Figure~\ref{fig:example} for an example). 
  Up to changing the orientation of $\gamma_1$ we may assume that the
  element induced by $c_0$ is $A_1$.

  We now describe a procedure that produces essential disks from
  admissible arcs. To this end, let $c$ be an admissible arc.
  Consider a regular neighborhood $U$ of $D_1\cup c\cup D_2$. Its
  boundary consists of three simple closed curves. Two of them are
  homotopic to either $\partial D_1$ or $\partial D_2$. The third one
  we denote by $\beta(c)$. Note that $\beta(c)$ bounds a nonseparating
  disk in $V$.

  Choose a fixed element $\varphi$ of the handlebody group of $V$ with
  the following properties. The mapping class $\varphi$ fixes the
  isotopy classes of the curves
  $\alpha_1,\alpha_2$ and $\alpha_3$. The restriction of $\varphi$ to 
  the complement of $V'$ is isotopic to the identity. 
  The restriction of
  $\varphi$ to $V'$ induces an automorphism of exponential growth type
  on $\pi_1(V')$.
  To be somewhat more precise, we may choose $\varphi$ such that it
  acts on the basis $A_i$ as the following automorphism $\Phi$:
  \begin{eqnarray*}
    A_1 &\mapsto& A_1A_2 \\
    A_2 &\mapsto& A_1^2A_2
  \end{eqnarray*}

  Put $c_n = \varphi^n(c_0)$ and $\beta_n = \beta(c_n)$.
  We claim that a disk exchange sequence in direction of $\beta_n$
  that makes $\beta_n$ disjoint 
  from $\Sigma_0$ has length at least $2^n$.

  To this end, note that the arc $c_n$ intersects the disks $D_3$ and
  $D_4$ in at least $2^n$ points. 
  Namely, the element of $\pi_1(V^\prime,p)$
  induced by $\varphi^n(c_0)$ is equal to
  $\Phi^n(A_1)$. The cyclically reduced word describing $\Phi^n(A_1)$
  in the basis $A_1, A_2$ has length at least $2^n$ by construction of $\Phi$.

  Therefore, the curve $\beta_n$ can be described as
  follows. Choose a parametrization $\beta_n:[0,1]\to \partial
  V$. Then there are numbers $0<t_1<\dots<t_N<t_{N+1}<\dots<t_{2N}<1$
  such that the following holds. Each subarc $\beta_n([t_i,t_{i+1}])$
  intersects $\Sigma_0$ only at its endpoints. The subarcs
  $\beta_n([t_N,t_{N+1}])$ and $\beta_n( [0,t_1] \cup [t_{2N},1] )$ are
  returning arcs to $\Sigma_0$.
  Furthermore, the arcs $\beta_n([t_i,t_{i+1}])$ and
  $\beta_n([t_{2N-i},t_{2N+1-i}])$ are homotopic relative to $\Sigma_0$
  for all $i=1,\ldots, N-1$.
  More generally, if there are numbers $t_i$ with these properties for
  a reduced disk system $\Sigma$ we say that \emph{$\beta_n$ is a long
    string of rectangles with respect to $\Sigma$}. The number $N$ is
  then called \emph{the length of the string of rectangles}.
  By construction, the length $N$ of the string of rectangles
  $\beta_n$ defines with respect to $\Sigma_0$ is at least $2^n$.

  The curve $\beta_n$ has two returning arcs. Let $a$ be one of
  them, say $\beta_n([t_N,t_{N+1}])$ and let $\sigma \in \Sigma_0$
  denote the disk containing the endpoints of $a$. 
  One of the disks obtained by simple surgery along $a$ is isotopic to
  either $D_1$ or $D_2$ (depending on which returning arc we chose).
  The preferred interval defined by $a$ contains every intersection
  point of $\beta_n$ with $\sigma$ except the endpoints of $a$.

  Denote by $\Sigma_1$ the reduced disk system obtained by
  simple surgery along $a$.
  By construction, the subarc $\beta_n(t_{N-1},t_{N+2})$ now defines a
  returning arc with respect to $\Sigma_1$. 
  One of the disks obtained by simple surgery along this returning arc
  is still properly isotopic to $D_1$.
  Furthermore, the subarcs
  $\beta_n([t_i,t_{i+1}])$ and $\beta_n([t_{2N-i},t_{2N+1-i}])$ are still
  arcs with endpoints on $\Sigma_1$ which are homotopic relative to $\Sigma_1$
  for all $i=1,\ldots, N-2$. Each of these arcs cannot be homotoped
  into $\partial \Sigma_1$.
  
  Hence the curve $\beta_n$ has a description as a string of
  rectangles of length $N-1$ with respect to $\Sigma_1$ and the
  argument can be iterated. By induction,
  it follows that any disk exchange sequence starting in $\Sigma_0$
  which ends in a disk system disjoint from $\beta_n$ has length at
  least $2^n$. 

  On the other hand, the growth of the distance between $\Sigma_0$ and
  $\varphi^n(\Sigma_0)$ in the graph of reduced disk systems is
  linear in $n$ by the triangle inequality. 
  The curve $\beta_n$ intersects $\varphi^n(\Sigma_0)$ in uniformly
  few points, and thus the disk system $\varphi^n(\Sigma_0)$ is
  uniformly close to a reduced disk system that is disjoint from
  $\beta_n$.
  Thus the disk systems $\Sigma_n$ have the properties described in
  the beginning of the example.
\end{example}


\begin{thebibliography}{MMS10}


%




\bibitem[BV95]{BV95} M.~Bridson, K.~Vogtmann, 
{\em On the geometry
of the group of automorphisms of a free group},
Bull. London Math. Soc. 27 (1995), 544--525.


\bibitem[BV10]{BV10} M.~Bridson, K.~Vogtmann, 
{\em The Dehn function of ${\rm Aut}(F_n)$ and
${\rm Out}(F_n)$}, arXiv:1011.1506.







%
%

\bibitem[H09b]{H09b} U.~Hamenst\"adt, {\em Geometry of the
mapping class group II: A biautomatic structure},
arXiv:0912.0137.



\bibitem[HH11]{HH11} U.~Hamenst\"adt, S.~Hensel, {\em The geometry of the
handlebody groups I: Distortion},
arXiv:1101.1838.


\bibitem[HM10]{HM10} M. Handel, L. Mosher,
{\em Lipschitz retraction and distortion for
subgroups of ${\rm Out}(F_n)$}, arXiv:1009.0518.

\bibitem[Ha95]{Ha95} A.~Hatcher, {\em Homological
stability for automorphism groups of free groups},
Comm. Math. Helv. 70 (1995), 39--62.

\bibitem[HV96]{HV96}
A.~Hatcher, K.~Vogtmann,
{\em Isoperimetric inequalities for automorphism groups of free groups},
Pacific J. Math. 173 (1996), no. 2, 425--441.



\bibitem[M86]{M86} H.~Masur, {\em Measured foliations
and handlebodies}, Erg. Th. \& Dynam. Sys. 6 (1986), 99--116.


\bibitem[MM00]{MM00} H. Masur, Y. Minsky, 
{\em Geometry of the complex of curves II: Hierarchical structure},
Geom. Funct. Anal. 10 (2000), 902--974.



\bibitem[McC85]{McC85} D.~McCullough,
{\em Twist groups of compact $3$--manifolds}, 
Topology 24 (1985), no. 4, 461--474.

\bibitem[Mo95]{Mo95} L.~Mosher,
{\em Mapping class groups are automatic}, 
Ann. of Math. (2) 142 (1995), no. 2, 303--384.


%

%

\bibitem[St99]{St02} J.~Stallings,
{\em Whitehead graphs on handlebodies},
Geometric group theory down under (Canberra, 1996), 317--330, de
Gruyter, Berlin, 1999.




\end{thebibliography}
\end{document}